\documentclass[12pt]{amsart}
\usepackage{amssymb,amsmath,amsfonts, amsthm, amscd}
\usepackage{arydshln}
\usepackage{float}
\usepackage{graphicx}
\usepackage{caption}
\usepackage{rotating}
\usepackage{pdflscape}
\usepackage{tikz}
\usepackage[utf8]{inputenc}
\usepackage{booktabs}
\usepackage{lipsum}
\usepackage{multirow}
\usepackage{enumitem}
\usepackage{subcaption}
\usepackage{multirow}
\usepackage{epsfig,amsthm,amsmath,amsfonts,amssymb}
\usepackage{fancyhdr,multicol,color,xcolor}
\usepackage[english]{babel}
\usepackage[utf8]{inputenc}
\usepackage{algorithm}
\usepackage[noend]{algpseudocode}
\numberwithin{equation}{section} \numberwithin{equation}{section}

\newcommand{\RNum}[1]{\uppercase\expandafter{\romannumeral #1\relax}}
\label{sec.sub.gdn}

\newtheorem{definition}{Definition}

\newtheorem{ex}{Example}
\newtheorem{theorem}{Theorem}
\newtheorem{proposition}{Proposition}
\newtheorem{remark}{Remark}

\numberwithin{equation}{section}

\begin{document}
	
\title{Solving Linear Systems over Tropical Semirings through Normalization Method and Its Applications}

\author{Fateme Olia}

\address{Fateme Olia, Faculty of Mathematics, K. N. Toosi University of Technology, Tehran, Iran}
\email{folya@mail.kntu.ac.ir}

\author{Shaban Ghalandarzadeh}

\address{Shaban Ghalandarzadeh, Faculty of Mathematics, K. N. Toosi University of Technology, Tehran, Iran}
\email{ghalandarzadeh@kntu.ac.ir}

\author{Amirhossein Amiraslani}

\address{Amirhossein Amiraslani,STEM Department, The University of Hawaii-Maui College, Kahului, Hawaii, USA. Faculty of Mathematics, K. N. Toosi University of Technology, Tehran, Iran}
\email{aamirasl@hawaii.edu}		
\author{Sedighe Jamshidvand}

\address{Sedighe Jamshidvand, Faculty of Mathematics, K. N. Toosi University of Technology, Tehran, Iran}
\email{sjamshidvand@mail.kntu.ac.ir}

	\begin{abstract}
		In this paper, we introduce and analyze a normalization method for solving a system of linear equations over tropical semirings. We use a normalization method to construct an associated normalized matrix, which gives a technique for solving the system. If solutions exist, the method can also determine the degrees of freedom of the system. Moreover, we present a procedure to determine the column rank and the row rank of a matrix. Flowcharts for this normalization method and its applications are included as well.
	\end{abstract}
	\maketitle 
	{\small {\it Key words}: Tropical semiring; system of linear equations; column rank; row rank}\\[0.3cm]
	{\bf Mathematics Subject Classification 2010:} 16Y60, 65F05, 15A03, 15A06.

\section{Introduction}
Systems of linear equations play a fundamental role in numerical simulations and formulization of mathematics and physics problems. Solving these systems is among the important tasks of linear algebra. There are numerous applications of linear  systems  over tropical semirings in various areas of mathematics, engineering, computer science, optimization theory, control theory, etc. (see e.g. \cite{Aceto}, \cite{Gondran},\cite{krivulin1}, \cite{krivulin2},\cite{McEneaney} ). As such, we intend to present a method for examining the behavior of linear systems and solving them if possible.\\
The algebraic structure of semirings are similar to rings, but subtraction and division can not necessarily be defined for them. The first notion of a semiring was given by Vandiver \cite{vandiver} in 1934. A semiring $(S,+,.,0,1)$ is an algebraic structure in which $(S, +)$ is a commutative monoid with an identity element $0$ and $(S,.)$ is a monoid with an identity element 1, connected by ring-like distributivity. The additive identity $0$ is multiplicatively absorbing, and $0 \neq 1$. Note that for convenience, we mainly consider $S=(\mathbb{R} \cup \{-\infty\}, max, +, -\infty, 0)$ which is a well-known tropical semiring called ``$\max-\rm plus$ algebra" in this work. Other examples of tropical semirings, which are isomorphic to ``$\max-\rm plus$ algebra", are ``$\max-\rm times$ algebra", ``$\min-\rm times$ algebra "and ``$\min-\rm plus$ algebra". \\
Letting $ S $ be a tropical semiring, we want to solve the system $ AX=b $, where $ A=(a_{ij}) \in M_{m\times n}(S) $, $ b \in S^{m} $ and $ X $ is an unknown vector of size $ n $. To this end, we present a necessary and sufficient condition based on the associated normalized matrix, which is obtained from a proposed normalization method. Furthermore, if the system $ AX=b $ has solutions, we use the associated normalized matrix to determine the degrees of freedom of the system. Determining the column rank and the row rank of a matrix is of particular interest in studying the behavior of matrices. As a result of the normalization method, we are able to find the column rank and the row rank of matrices over tropical semirings.\\
In section~3, by introducing the normalization method, we can construct the associated normalized matrix of a linear system that provides useful information about this system. As an extremely important result, we present a necessary and sufficient condition on the associated normalized matrix to determine the existence of solutions of the system. Additionally, we introduce an equivalent relation over matrices that implies the associated normalized matrix of a linear system and each of its equivalent systems should be the same. As such, the solvability of a linear system and its equivalent system depend on each other.\\
Section~4 concerns a descriptive method for finding the degrees of freedom under a step-by-step process on the associated normalized matrix.\\
In section~5, we determine the column rank of a matrix and the row rank as the column rank of the matrix transpose by the normalization method. In fact, we must investigate the dependence of each column on other columns of the matrix. To this end, we first solve a linear system, where the system vector can be each column of the matrix and other columns form the system matrix. The solvability of this system means that the system vector is linearly dependent and it must be removed. We repeat this process for each column. \\
Through row-column analysis, we can remove linearly dependent rows and columns of the system matrix to obtain the reduced system with fewer equations and unknowns. It is shown that the solvability of a linear system and its corresponding reduced system depend on each other. Moreover, if a linear system has a solution, then its corresponding reduced system can accelerate the computation of degrees of freedom, which is defined as the number of free variables of the system. See section~6 for more details.\\
In the appendix of this paper, we give some flowcharts as follows. Figure~1 shows the normalization process of a matrix (or vector) in max-plus. Figure~2 gives the method for solving a system of linear equations and finding its degrees of freedom in max-plus. Finally, Figure~3 is about finding the column rank of a given matrix in max-plus.
\section{Definitions and Preliminaries}
In this section, we give some definitions and preliminaries. For convenience, we use $ \mathbb{N} $ and $\underline{n}$ to denote the set of all positive integers and the set $\{ 1,2,\cdots,n\}$ for $n \in \mathbb{N}$, respectively.
\begin{definition}
	A semiring $ (S,+,., 0,1)$ is an algebraic system consisting of a nonempty set $ S $ with two binary operations,  addition and multiplication, such that the following conditions hold:
	\begin{enumerate} 
		\item{$ (S, +) $ is a commutative monoid with identity element $ 0 $};
		\item{ $ (S, \cdot) $ is a monoid with identity element $ 1 $};
		\item{ Multiplication distributes over addition from either side, that is $ a(b+c)= ab+ac $ and $ (b+c)a=ba+ca $ for all $ a, b \in S $};
		\item{ The neutral element of $ S $ is an absorbing element, that is $ a\cdot 0 =0= 0 \cdot a  $ for all $ a \in S $};			\item{ $ 1 \neq 0 $}.
	\end{enumerate}
	A semiring is called commutative if $ a\cdot b = b \cdot a $ for all $ a, b \in S $.
\end{definition}
In this work, we primarily focus on tropical semiring $\mathbb{R}_{\max, +}:=(\mathbb{R}\cup \{ -\infty \}, \max, +, -\infty, 0)$, which is called ``$\max-\rm plus$ algebra" whose additive and multiplicative identities are $-\infty$ and $0$, respectively. Moreover, the notation $a-b$ in ``$\max-\rm plus$ algebra" is equivalent to $ a + (-b)  $, where $ ``-" ,~``+" $ and $ -b $ denote the usual real numbers subtraction, addition and the typical additively inverse of the element $ b $, respectively.
\begin{definition}(See \cite{golan})
	Let $ S $ be a semiring. A left $ S $-semimodule is a commutative monoid $ (\mathcal{M},+) $ with identity element $ 0_{\mathcal{M}}$ for which we have a scalar multiplication function $S \times \mathcal{M} \longrightarrow \mathcal{M} $, denoted by $ (s, m)\mapsto sm $, which satisfies the following conditions for all $ s, s^{\prime} \in S $ and $ m, m^{\prime} \in \mathcal{M} $:
	\begin{enumerate}
		\item{$ (ss^{\prime})m=s(s^{\prime}m) $ }; 
		\item{$ s(m+m^{\prime})= sm +sm^{\prime}$};
		\item{$ (s+s^{\prime})m=sm +s^{\prime}m $};
		\item{$ 1_{S}m=m$};
		\item{$ s0_{\mathcal{M}}=0_{\mathcal{M}}=0_{S}m$}.
	\end{enumerate}
	Right semimodules over $ S $ are defined in an analogous manner. 
\end{definition}
\begin{definition}
	A nonempty subset $ \mathcal{N} $ of a left $ S$-semimodule $ \mathcal{M} $ is a subsemimodule of $ \mathcal{M}$ if $ \mathcal{N} $ is closed under addition and scalar multiplication. Note that this implies $ 0_{\mathcal{M}} \in \mathcal{N} $. Subsemimodules of right semimodules are defined analogously. 
\end{definition}
\begin{definition}
	Let $ \mathcal{M} $ be a left $ S $-semimodule and $ \{ \mathcal{N}_{i} \vert i \in \Omega \} $ be a family of subsemimodules of $ \mathcal{M} $. Then $ \displaystyle{\bigcap_{i \in \Omega}}\mathcal{N}_{i}$ is a subsemimodule of $ \mathcal{M} $ which, indeed, is the largest subsemimodule of $ \mathcal{M} $ contained in each of the $ \mathcal{N}_{i}$. In particular, if $ \mathcal{A} $ is a subset of a left $ S $-semimodule $ \mathcal{M} $ then the intersection of all subsemimodules of $ \mathcal{M} $ containing $ \mathcal{A} $ is a subsemimodule of $ \mathcal{M} $, called the subsemimodule generated by $ \mathcal{A} $. This subsemimodule is denoted by
	$$ S\mathcal{A}=Span(\mathcal{A}) = \{\displaystyle{\sum_{i=1}^{n}}s_{i}\alpha_{i}~ \vert~ s_{i} \in S, \alpha_{i} \in \mathcal{A},i \in \underline{n}, n \in \mathbb{N}  \}.$$
	If $  \mathcal{A}$ generates all of the semimodule $ \mathcal{M} $ , then $\mathcal{A} $ is a set of generators for $ \mathcal{M} $. Any set of generators for $ \mathcal{M} $ contains a minimal set of generators. A left $ S $-semimodule having a finite set of generators is finitely generated. Note that the expression $ \displaystyle{\sum_{i=1}^{n}}s_{i}\alpha_{i} $ is a linear combination of the elements of $ \mathcal{A} $.
\end{definition}
\begin{definition}(See \cite{taninner})
	Let $\mathcal{M}$ be a left $ S $-semimodule. A nonempty subset $ \mathcal{A}$ of $\mathcal{M} $ is called linearly independent if $ \alpha \notin Span(\mathcal{A} \setminus \{ \alpha \}) $ for any $ \alpha \in \mathcal{A} $. If $\mathcal{A}$ is not linearly independent then it is called linearly dependent.
\end{definition}
\begin{definition}
	The rank of a left $ S $-semimodule $ \mathcal{M} $ is the smallest $ n $ for which there exists a set of generators  of $ \mathcal{M} $ with cardinality $ n $. It is clear that  $ rank(\mathcal{M})$ exists for any finitely generated left $ S $-semimodule $ \mathcal{M} $. \\
	This rank need not be the same as the cardinality of a minimal set of generators for $ \mathcal{M} $, as the following example shows. 
\end{definition}
\begin{ex}
	Let $S $ be a semiring and $\mathcal{R} =S \times S $ be the Cartesian product of two copies of $ S $. Then $ \{ (1_{S},1_{S}) \}$ and $ \{ (1_{S}, 0_{S}), (0_{S},1_{S}) \}$ are both minimal sets of generators for $ \mathcal{R} $, considered as a left semimodule over itself with componentwise addition and multiplication. Hence, $ rank(\mathcal{R})= 1 $.
\end{ex}
Let $S$ be a commutative semiring. We denote the set of all $m \times n$ matrices over $S$ by $M_{m \times n}(S)$.  
For $A \in M_{m \times n} (S)$, we denote by $a_{ij}$ and $A^{T}$ the $(i,j)$-entry of $A$ and the transpose of $A$, respectively.\\
For any $A, B \in M_{m \times n}(S)$, $C \in M_{n \times l}(S)$ and $\lambda \in S$, we define:
$$ A+B = (a_{ij}+b_{ij})_{m \times n},$$
$$ AC=(\sum_{k=1}^{n} a_{ik}b_{kj})_{m \times l},$$
and $$\lambda A=(\lambda a_{ij})_{m\times n}.$$
Clearly, $ M_{m \times n}(S) $ equipped with matrix addition and matrix scalar multiplication is a left $ S $-semimodule.
It is easy to verify that $M_{n}(S):=M_{n \times n}(S)$ forms a semiring with respect to the matrix addition and the matrix multiplication.
The above matrix operations over $ \max-\rm plus$ algebra can be considered as follows.
$$A+B = (\max(a_{ij},b_{ij}))_{m \times n},$$ 
$$AC=(\max_{k=1}^{n} (a_{ik}+b_{kj}))_{m \times l},$$ 
and $$\lambda A=(\lambda + a_{ij})_{m\times n}.$$
For convenience, we can denote the scalar multiplication $ \lambda A $ by $ \lambda + A $. Moreover, $ \max-\rm plus$ algebra is a commutative semiring, which implies $ \lambda + A= A +\lambda $.
\begin{definition}
	Let $A, B \in M_{n}(S)$ such that $A=(a_{ij})$ and $B=(b_{ij})$. We say $A\leq B$ if and only if $a_{ij} \leq b_{ij}$ for every $i\in \underline{m}$ and $j \in \underline{n}$. 
\end{definition}
\begin{definition}(See \cite{wilding})
	Let $ S $ be a semiring and $A \in M_{m \times n}(S) $. The column space of $ A $ is the finitely generated right $ S $-subsemimodule of $M_{m \times 1}(S) $ generated by the columns of $A$:
	$$ Col(A)=\{ Av \vert v \in M_{n \times 1}(S) \}. $$
	The column rank of $A $ is the rank of its column subsemimodule, which is denoted by $colrank(A)$.
\end{definition}
\begin{definition}(See \cite{wilding})
	Let $ S $ be a semiring and $A \in M_{m \times n}(S) $. The row space of $ A $ is the finitely generated left $ S $-subsemimodule of $M_{1 \times n}(S) $ generated by the rows of $A$:
	$$ Row(A)=\{ uA \vert u \in M_{1 \times m}(S) \}. $$
	The row rank of $A $ denoted by $rowrank(A)$ is the rank of its row subsemimodule.
\end{definition}
The next example shows that the column rank and the row rank of a matrix over an arbitrary semiring are not necessarily equal. If these two value  coincide, their common value is called the rank of matrix $ A $.
\begin{ex}
	Consider $A \in M_{3}(S) $ where $ S=\mathbb{R}_{\max,+}$ as follows.
	\[
	A=\left[
	\begin{array}{ccc}
	3&6&5\\
	-5&0&-2\\
	4&1&6
	\end{array}
	\right].
	\]
	Clearly, $ rowrank(A)=3 $, but $colrank(A)=2$, since  the third column of $ A $ is a linear combination of its other columns:
	$$ A_{3}=\max(A_{1}+ 2,A_{2}+(-2)).$$
\end{ex}
Next, we study and analyze the system of linear equations $ AX=b $ where $A \in M_{m\times n}(S) $, $ b \in S^{m} $ and  $ X $ is an unknown column vector of size $n$ over tropical semiring $ S=\mathbb{R}_{\max,+}$, whose $i-$th equation is
$$\max(a_{i1}+ x_{1}, a_{i2} +x_{2}, \cdots, a_{in} +x_{n})=b_{i}. $$
Sometimes, we can simplify the solution process of the system, $ AX=b $, by turning that into a linear system of equations with fewer equations and variables. 
\begin{definition}
	Let $A \in M_{m \times n}(S) $. A reduced matrix is obtained from matrix $ A $ by removing its  dependent rows and columns which we denote by $ \overline{A}$. 
\end{definition}
\begin{definition}
	A solution $X^{*}$ of the system $AX=b$ is called maximal, if $X \leq X^{*}$ for any solution $X$.
\end{definition}
\begin{definition}
	A vector $b\in S^{m}$ is called  regular if $ b_{i} \neq -\infty $ for any $i \in \underline{m}$.
\end{definition}
Without loss of generality, we can assume that $b$ is regular in the system $AX=b$. Otherwise, let $b_{i}=-\infty $ for some $i \in \underline{n}$. Then in the $i-$th equation of the system, we have $a_{ij} + x_{j}= -\infty $ for any $j \in \underline{n}$. As such, $x_{j}= -\infty $ if $a_{ij}\neq -\infty $. Consequently, the $i-$th equation can be removed from the system together with every column $A_{j}$ where $a_{ij} \neq -\infty $, and the corresponding $x_{j}$ can be set to $-\infty $.
\begin{definition}
	Let the linear system of equations $ AX=b $ have solutions. Suppose that $ A_{j_{1}}, A_{j_{2}}, \cdots, A_{j_{k}} $ are linearly independent columns of $ A $, and $ b $ is a linear combination of them. Then the corresponding variables, $ x_{j_{1}}, x_{j_{2}}, \cdots, x_{j_{k}} $ , are called leading variables and other variables are called free variables of the system $ AX= b$. \\
	The degrees of freedom of the linear system  $ AX=b $, denoted by $ \mathcal{D}_{f}$, is the number of free variables. 
\end{definition}
\section{Solving a System of Linear Equations through the Normalization Method}
In this section, we introduce a method, which we call the normalization method, for solving a system of linear equations.
Consider the system of linear equations $AX=b$, where $ A=(a_{ij})\in M_{m \times n}(S) $, $ b=(b_{i}) $ is a regular $m-$vector over $ S $ and $ X $  is an unknown $ n-$vector. Let the $ j $-th column of the matrix $ A $ be denoted by $ A_{j} $ . 
\begin{definition}\label{normmethod} \textbf{(Normalization Method)} Let $ A \in M _{m \times n}(S) $ and $A_{j} \in S^{m} $ be a regular vector for any $ j \in \underline{n} $ which means the matrix $ A $ does not contain any element $ -\infty $. Then the normalized matrix of $ A $ is denoted by
	\begin{align*}
	\tilde{A}=
	\left[
	\begin{array}{c|c|c|c}
	A_{1}- \hat{A}_{1}&A_{2} - \hat{A}_{2}&\cdots&A_{n} - \hat{A}_{n}
	\end{array}
	\right], 
	\end{align*}
	where $ \hat{A}_{j}= \frac{a_{1j}+ a_{2j} + \cdots + a_{mj}}{m}$ for every $ j \in \underline{n}$.\\
	Similarly, the normalized vector of the regular vector $ b\in S^{m} $ is 
	$$\tilde{b} = b - \hat{b},$$ 
	where  $ \hat{b}= %\sqrt[m]{b_{1}\otimes b_{2}\otimes \cdots \otimes b_{m}} =
	\frac{b_{1}+ b_{2} + \cdots + b_{m}}{m} $.\\
	As such, we can rewrite the system $ AX=b $ as the normalized system $ \tilde{A}Y=\tilde{b}$, where $ Y=(\hat{A}_{j} - \hat{b})+X=(\hat{A}_{j} - \hat{b} +x_{j})_{j=1}^{n} $, as follows.
	\begin{align*}
	AX=b
	&\ \Rightarrow\max(A_{1} + x_{1}, A_{2} + x_{2}, \cdots, A_{n} + x_{n} ) =b \\
	&\ \Rightarrow \max((A_{1}- \hat{A}_{1} )+\hat{A}_{1} + x_{1},( A_{2}- \hat{A}_{2} )+\hat{A}_{2} + x_{2}, \cdots , (A_{n}- \hat{A}_{n} )+\hat{A}_{n}+ x_{n} ) =(b -\hat{b}) +\hat{b}\\
	&\ \Rightarrow \max(\tilde{A}_{1}+\hat{A}_{1} + x_{1},\tilde{A}_{2}+\hat{A}_{2} + x_{2}, \cdots ,\tilde{A}_{n} +\hat{A}_{n}+ x_{n} ) =\tilde{b} +\hat{b}\\ 
	&\ \Rightarrow \max(\tilde{A}_{1}+(\hat{A}_{1}- \hat{b} + x_{1}),\tilde{A}_{2}+(\hat{A}_{2}- \hat{b} + x_{2}), \cdots ,\tilde{A}_{n} +(\hat{A}_{n} - \hat{b}+ x_{n}) ) =\tilde{b}\\ 
	&\ \Rightarrow \max(\tilde{A}_{1}+y_{1},\tilde{A}_{2}+y_{2}, \cdots ,\tilde{A}_{n} +y_{n} ) =\tilde{b}\\
	&\ \Rightarrow \tilde{A}Y=\tilde{b}
	\end{align*}
	Hence $ y_{j} \leq \tilde{b}_{i} - \tilde{a}_{ij} $ for every $ i \in \underline{m}$ and $j \in \underline{n} $. Now, we define the associated normalized matrix $ Q= (q_{ij})\in M_{m\times n}(S) $ where $ q_{ij}=\tilde{b}_{i} -\tilde{a}_{ij} $ . We choose $y_{j} $ as the minimum element of $Q_{j}$ (the $j$-th column of $Q$), which we call the ``$j$-th column minimum element".\\
	It should be noted that if $ a_{ij} = -\infty $ for some $ i \in \underline{m} $ and $ j \in \underline{n} $, then we will not count $ a_{ij} $ in the normalization process of column $ A_{j} $, i.e. 
	$$ \hat{A}_{j}=\frac{a_{1j}+ a_{2j} + \cdots + a_{(i-1)j}+ a_{(i+1)j}+ \cdots + a_{mj}}{m-1}. $$ 
	As such, $ \tilde{a}_{ij}= -\infty $ and  we set $ q_{ij}:=(-\infty)^{-} $ such that $ s < (-\infty)^{-} $ for any $ s \in S $. Thus, $ q_{ij}$ does not affect the $j-$th column minimum element. Consequently and without loss of generality, we assume that every column of the system matrix is regular.
\end{definition}
\begin{remark}
	Normalization method is a useful computational method in which rewriting the system $ AX=b $ as the normalized system $ \tilde{A}Y=\tilde{b} $ provides an appropriate criterion to compare the entries of each column. For instance, in the system $ \tilde{A}Y=\tilde{b} $, the multiplication of all nonzero elements in every column of matrices $ \tilde{A}$ and $\tilde{b} $ is equal to $ 1 $. In max-plus algebra, it means $ \displaystyle{\sum_{i=1}^{m} \tilde{a}_{ij}} = 0$ and $\displaystyle{\sum_{i=1}^{m} \tilde{b}_{i}} = 0 $, for any $ \tilde{a}_{ij}, \tilde{b}_{i} \neq -\infty $ and $ 1 \leq i, j \leq n $ with $ A, \tilde{A} \in M_{m \times n}(S) $. Note further that by constructing the associated normalized matrix $ Q $ as given by Definition~\ref{normmethod} and determining its column minimum elements, some comprehensive information about the existence of solutions of the system $ AX=b $, solving this system and every equivalent system to $ AX=b $ is obtained from the matrix $ Q $. The normalization method can also be used to determine the degrees of freedom of a solvable system as well as the column rank and the row rank of a given matrix.
\end{remark}
In the next theorem, we give a necessary and sufficient condition for solving the system $AX=b$ such that the matrices $ A, Q$ and the vector $ b $ are defined as above.
\begin{theorem}\label{normthm}
	The linear system of equations $ AX=b$ has solutions if and only if there exists at least one column minimum element in every row of $Q$.
\end{theorem}
\begin{proof}
	Let the system $AX=b$ has solutions. Suppose the $i$-th row of $ Q $ has no column minimum element for some $i \in \underline{m}$. That is $ y_{j} < \tilde{b}_{i} - \tilde{a}_{ij} $ for every $ j \in \underline{n} $, therefore the $i$-th equation of the system $ \tilde{A}Y=\tilde{b} $ is 
	$$ \max(\tilde{a}_{i1} + y_{1}, \tilde{a}_{i2}+ y_{2}, \cdots,\tilde{a}_{in}+ y_{n} ) < \tilde{b}_{i}.$$
	Hence, the system $\tilde{A}Y=\tilde{b} $ and a fortiori the system $ AX=b $ have no solution, which is a contradiction.\\
	Conversely, suppose that every row of the matrix $ Q $ contains at least one column minimum element, so for any $ i \in \underline{m} $ there is some $ j \in \underline{n} $ such that $ y_{j}= \tilde{b}_{i} - \tilde{a}_{ij}$. Then 
	$$ \max(\tilde{a}_{i1} + y_{1}, \tilde{a}_{i2}+ y_{2}, \cdots,\tilde{a}_{ij}+ y_{j} ,\cdots,\tilde{a}_{in}+ y_{n} ) = \tilde{b}_{i} $$ for every $ i \in \underline{m}$. Thus, the system $ \tilde{A}Y=\tilde{b} $ and consequently the system $AX=b $ have solutions.  
\end{proof}
\begin{remark}
	The solution of the system $ AX=b $ that is obtained from Theorem~\ref{normthm} is maximal.
\end{remark}
The next example shows that the condition of Theorem~\ref{normthm} is a sufficient condition for a linear system of equations to have solutions.
\begin{ex}\label{normex}
	Let $ A \in M_{4 \times 5}(S) $. Consider the following system $ AX=b $:
	\[
	\left[
	\begin{array}{ccccc}
	165&57&72&-7&0\\
	141&64&48&3&-1\\
	137&101&46&0&2\\
	-243&98&-206&156&-5
	\end{array}
	\right]
	\left[
	\begin{array}{c}
	x_{1}\\
	x_{2}\\
	x_{3}\\
	x_{4}\\
	x_{5}
	\end{array}
	\right]=
	\left[
	\begin{array}{c}
	102\\
	78\\
	76\\
	160
	\end{array}
	\right].	
	\]
	By Definition~\ref{normmethod}, the system $ AX= b $ is rewritten as the normalized system $ \tilde{A}Y=\tilde{b}$:
	\[
	\left[
	\begin{array}{ccccc}
	115&-23&82&-45&1\\
	91&-16&58&-35&0\\
	87&21&56&-38&3\\
	-293&18&-196&118&-4
	\end{array}
	\right]
	\left[
	\begin{array}{c}
	y_{1}\\
	y_{2}\\
	y_{3}\\
	y_{4}\\
	y_{5}
	\end{array}
	\right]=
	\left[
	\begin{array}{c}
	-2\\
	-26\\
	-28\\
	56
	\end{array}
	\right].	
	\]
	Note that the $ j $-th column of $ \tilde{A} $ is $\tilde{A}_{j}=(a_{ij} - \hat{A}_{j})_{i=1}^{4}$, for any $ 1 \leq j \leq 5 $ and $ \tilde{b}=(b_{i} -\hat{b})_{i=1}^{4} $, where $ \hat{A}_{1}= 50,~ \hat{A}_{2}= 80,~ \hat{A}_{3}= -10,~ \hat{A}_{4}= 38,~ \hat{A}_{5}= -1,~ \hat{b}= 104 $. Now, we can build the matrix $ Q= (q_{ij})\in M_{4 \times 5}(S) $, with $ q_{ij}= \tilde{b}_{i}- \tilde{a}_{ij}$ as follows.
	\[
	\left[
	\begin{array}{ccccc}
	\boxed{-117}&21&\boxed{-84}&43&-3\\
	\boxed{-117}&-10&\boxed{-84}&9&-26\\
	-115&\boxed{-49}&\boxed{-84}&10&\boxed{-31}\\
	349&38&252&\boxed{-62}&60
	\end{array}
	\right];
	\]
	where the minimum column elements are boxed. Since every row of $ Q $ contains at least one of these minimum column elements, due to Theorem~\ref{normthm}, the system $ \tilde{A}Y= \tilde{b} $ has the maximal solution $ Y^{*} $:
	\[ 
	Y^{*}=\left[
	\begin{array}{c}
	-117\\
	-49\\
	-84\\
	-62\\
	-31
	\end{array}
	\right]
	\] 
	Hence, the system $ AX=b $ has the maximal solution $ X^{*} $:
	\[
	X^{*}=\left[
	\begin{array}{c}
	-63\\
	-25\\
	30\\
	4\\
	74
	\end{array}
	\right];
	\]
	where $ x^{*}_{j}= y^{*}_{j} - \hat{A_{j}} + \hat{b} $, for any $ 1 \leq j \leq 5 $.
\end{ex}
The following example shows that the condition of Theorem~\ref{normthm} is  necessary. 
\begin{ex}
	Let $ A \in M_{5 \times 4}(S) $. Consider the following system $ AX=b $:
	\[
	\left[
	\begin{array}{cccc}
	0&-1&2&7\\
	1&5&4&-2\\
	-2&5&0&2\\
	4&-3&1&2\\
	-3&8&2&-6
	\end{array}
	\right]
	\left[
	\begin{array}{c}
	x_{1}\\
	x_{2}\\
	x_{3}\\
	x_{4}
	\end{array}
	\right]=
	\left[
	\begin{array}{c}
	3\\
	3\\
	0\\
	-6\\
	2
	\end{array}
	\right].	
	\]
	By Definition~\ref{normmethod}, the normalized system $ \tilde{A}Y= \tilde{b} $	corresponding to the system $ AX= b $ is as follows:
	\[
	\left[
	\begin{array}{cccc}
	0&-\frac{19}{5}&\frac{1}{5}&\frac{32}{5}\\
	1&\frac{11}{5}&\frac{11}{5}&-\frac{13}{5}\\
	-2&\frac{11}{5}&-\frac{9}{5}&\frac{7}{5}\\
	4&-\frac{29}{5}&-\frac{4}{5}&\frac{7}{5}\\
	-3&\frac{26}{5}&\frac{1}{5}&-\frac{33}{5}
	\end{array}
	\right]
	\left[
	\begin{array}{c}
	y_{1}\\
	y_{2}\\
	y_{3}\\
	y_{4}
	\end{array}
	\right]=
	\left[
	\begin{array}{c}
	\frac{13}{5}\\
	\frac{13}{5}\\
	-\frac{2}{5}\\
	-\frac{32}{5}\\
	\frac{8}{5}
	\end{array}
	\right],	
	\]
	where $ \hat{A}_{1}=0,~ \hat{A}_{2}= \frac{14}{5},~ \hat{A}_{3}= \frac{9}{5},~ \hat{A}_{4}= \frac{3}{5},~ \hat{b}= \frac{2}{5} $. Obviously, some rows of the following matrix $ Q= (q_{ij})\in M_{5 \times 4}(S) $, with $ q_{ij}= \tilde{b}_{i}- \tilde{a}_{ij}$ contain no column minimum element: 
	\[
	Q=\left[
	\begin{array}{cccc}
	\frac{13}{5}&\frac{32}{5}&\frac{12}{5}&-\frac{19}{5}\\
	\frac{8}{5}&\frac{2}{5}&\frac{2}{5}&\frac{26}{5}\\
	\frac{8}{5}&-\frac{13}{5}&\frac{7}{5}&-\frac{9}{5}\\
	\boxed{-\frac{52}{5}}&-\frac{3}{5}&\boxed{-\frac{28}{5}}&\boxed{-\frac{39}{5}}\\
	\frac{23}{5}&\boxed{-\frac{18}{5}}&\frac{7}{5}&\frac{41}{5}
	\end{array}
	\right].
	\]
	As such, by Theorem~\ref{normthm}, the system $ AX=b $ has no solution. Indeed, considering $ Y=(-\frac{52}{5}, -\frac{18}{5}, -\frac{28}{5}, -\frac{39}{5})^{T} $ implies $ X=(-10, -6,-7, -8 )^{T} $, but it is not a solution of the system $ AX=b $. For instance, in the first equation of the system by replacing $ X $ we have:
	$$ \max(a_{11} + (-10), a_{12} +(-6), a_{13} + (-7), a_{14} +(-8) )= -1 \neq b_{1} $$ 
\end{ex}
\subsection{Solving equivalent systems of linear equations}
\begin{definition}
	Let  $A, A^{\prime} \in M_{m\times n}(S)$.We say $A$ is equivalent to $A^{\prime}$ if there exist nonzero coefficients $\alpha_{1}, \alpha_{2}, \cdots, \alpha_{n} \in S $ such that $A^{\prime}_{j}=A_{j}+ \alpha_{j}$ for any $j \in \underline{n}$, and we write 
	\begin{center}
		$ A \sim A^{\prime} \Longleftrightarrow  A^{\prime}= [ A_{1}+\alpha_{1}| \cdots | A_{n}+\alpha_{n} ]$ 
	\end{center}
	for some $\alpha_{1}, \alpha_{2}, \cdots, \alpha_{n} \in S \backslash \{ -\infty\}$.\\
	The equivalence class of $A$ is defined as follows.
	\begin{center}
		$[A]= \{ A^{\prime} \in M_{m \times n}(S) | A \sim A ^{\prime}\}$
	\end{center}
	Note that this equivalence relation also holds for vectors.
\end{definition}
In the next theorem, we prove that the solvability of the equivalent systems depend on each other.
\begin{theorem}\label{equthm}
	Let $A \in M_{m \times n}(S)$ and $b \in S^{m}$ be a regular vector. Then the system $AX=b$ has solutions if and only if the equivalent system $A^{\prime}X^{\prime}= b^{\prime}$ has solutions for any $A^{\prime} \in [A]$ and $b ^{\prime} \in [b]$.
\end{theorem}
\begin{proof}
	Suppose $AX=b$ has solutions. By theorem~\ref{normthm}, every row of its associated normalized matrix, $Q=(q_{ij})$, contains at least one column minimum element, where
	\begin{center}
		$q_{ij}= \tilde{b_{i}}- \tilde{a}_{ij}= (b_{i}-\hat{b})- (a_{ij}- \hat{A}_{j}).$ 
	\end{center}
	On the other hand, since $A^{\prime}=(a^{\prime}_{ij})\in [A]$ and $b^{\prime}=(b^{\prime}_{i}) \in [b]$, there exist coefficients $\alpha_{1},\alpha_{2},\cdots, \alpha_{n},\beta \in S \backslash \{ -\infty\}$ such that $a^{\prime}_{ij}=a_{ij}+\alpha_{j}$ and $b^{\prime}_{i}=b_{i}+\beta$. Now, consider the associated normalized matrix $Q^{\prime}=(q^{\prime}_{ij})$ of the system $ A^{\prime}X^{\prime}= b^{\prime} $ such that
	\begin{align*}
	q^{\prime}_{ij}= \tilde{b^{\prime}}_{i} - \tilde{a^{\prime}}_{ij}
	&\ = (b^{\prime}-\hat{b^{\prime}})- (a^{\prime}_{ij}- \hat{A^{\prime}}_{j})\\
	&\ = (b_{i}+\beta - \hat{b^{\prime}})-(a_{ij}+\alpha_{j}- \hat{A^{\prime}}_{j})\\
	&\ = (b_{i}-\hat{b})-(a_{ij}-\hat{A}_{j})~~~~~~~~~~~~~~~~~~~~~~~~~~~~~~~~~~~~~~~~~~~~~~(3.1)\\
	&\ = q_{ij},
	\end{align*}
	for any $ i \in \underline{m} $ and $ j \in \underline{n} $. It should be noted that the equality $(3.1)$ is obtained from:\\
	\begin{align*}
	\hat{b^{\prime}}=\frac{b^{\prime}_{1}+\cdots+b^{\prime}_{m}}{m}
	&\ = \frac{(b_{1}+\beta)+\cdots+(b_{m}+\beta)}{m}\\
	&\ = \frac{( b_{1}+\cdots+b_{m})}{m}+ \beta\\
	&\ = \hat{b}+\beta
	\end{align*}
	and 
	\begin{align*}
	\hat{A^{\prime}}_{j}= \frac{a^{\prime}_{1j}+\cdots+a^{\prime}_{mj}}{m}
	&\ = \frac{(a_{1j}+\alpha_{j})+\cdots+(a_{mj}+\alpha_{j})}{m}\\
	&\ = \frac{a_{1j}+\cdots+a_{mj}}{m}+ \alpha_{j}\\
	&\ = \hat{A}_{j}+\alpha_{j}.
	\end{align*}
	This means $ Q=Q^{\prime}$ and consequently, the column minimum elements of $Q$ and $Q^{\prime} $ are the same. Hence, the proof is complete. Similarly, we can prove the converse.
\end{proof}
\begin{remark}
	The proof of Theorem~\ref{equthm} shows that the associated normalized matrix of the system $ AX=b $ and each of its equivalent systems are the same. As such, if the system $ AX=b $ has the solution $ X=(x_{j})_{j=1}^{n} $, then we can find the solution of the equivalent system $A^{\prime}X^{\prime}=b^{\prime}$ as $ X^{\prime}=(x^{\prime}_{j})_{j=1}^{n}$ with $ x^{\prime}_{j}= x_{j} +\beta - \alpha_{j} $ for any $ j \in \underline{n}$.
\end{remark}
\section{Determining the Leading and Free Variables by the Normalization Method}
Consider the system $ AX=b $ and the matrix $ Q $ as given by Definition~\ref{normmethod}. Suppose that the system $ AX=b $ has solutions. Then by theorem~\ref{normthm}, any row of the matrix $ Q $ contains at least one column minimum element. If there exists exactly one  column minimum element in some rows of $ Q $ that is on the $ j$-th column of $ Q $, then $ y_{j} $ must be a leading variable of the system $\tilde{A}Y=\tilde{b} $. As such, the corresponding variable $x_{j} $ is a leading variable of the system $ AX=b $.
\begin{proposition}
	Let $ AX=b $ and the matrix $ Q $ be given by Definition~\ref{normmethod}. Suppose that the system $ AX=b $ has solutions and $ k $ is the number of the rows of $ Q $ which contain exactly one column minimum element in different columns. Then the following statements hold:
	\begin{enumerate}
		\item{ If $ k=0 $, then $ \mathcal{D}_{f} \leq n-1 $ }
		\item{ If $ k\neq 0$, then $ \mathcal{D}_{f} \leq n-k $ }
	\end{enumerate}
\end{proposition}
\begin{proof}
	The proof is obvious.
\end{proof}
\subsection{A descriptive method for finding the number of degrees of freedom}
Let the non negative integer $ k $ be the number of the rows of $ Q $ containing exactly one column minimum element in different columns.
\begin{itemize}
	\item \textbf{Step 1.} First, we determine the rows of $ Q $ which contain exactly one column minimum element. We now consider the columns of $ Q $ where these column minimum elements are located. The corresponding variables of these columns are leading variables of the system $ \tilde{A}Y=\tilde{b}$. Hence, the system has at least $ k $ leading variables. For example, suppose that a row of $ Q $ contains exactly one column minimum element that is located in the $ j$-th column. Then $ y_{j} $ and consequently $ x_{j} $ are leading variables of the systems $ \tilde{A}Y=\tilde{b} $ and $ AX=b $, respectively. 
	\item \textbf{Step 2.} Next, we remove every row of $ Q $ containing exactly one column minimum element and determine their column indices. We then eliminate the rows of the matrix $ Q $ whose column minimum elements occur in the same column index as the rows containing exactly one column minimum element. 
	\item \textbf{Step 3.} In the remaining rows from Step 2, we select the column whose column minimum elements appear most frequently (say, the $l $-th column). We consider the corresponding variable to this column as the next leading variable ($ x_{l}$). We now remove all the rows including $ x_{l} $. 
	\item \textbf{Step 4.} We now repeat Step 3 and continue until we remove all the rows of $ Q $. Eventually, we can obtain the total number of leading variables and the degrees of freedom which satisfy the following equation $$ \mathcal{D}_{f}= n- (the~number~of~leading~variables)$$
\end{itemize}
In the following two examples, we apply the above method to find the number of degrees of freedom of solvable linear systems.
\begin{ex}
	Let $ A \in M_{4 \times 5}(S)$. Consider the following system $ AX=b$:
	\[
	\left[
	\begin{array}{ccccc}
	-4&7&12&-3&0\\
	3&2&8&3&-1\\
	-9&1&6&0&2\\
	2&8&-5&1&-3
	\end{array}
	\right]
	\left[
	\begin{array}{c}
	x_{1}\\
	x_{2}\\
	x_{3}\\
	x_{4}\\
	x_{5}
	\end{array}
	\right]=
	\left[
	\begin{array}{c}
	5\\
	10\\
	4\\
	9
	\end{array}
	\right].
	\]
	By Definition~\ref{normmethod}, the normalized system $ \tilde{A}Y=\tilde{b} $ corresponding to the system $ AX=b $ is 
	\[
	\left[
	\begin{array}{ccccc}
	-2&\frac{5}{2}&\frac{27}{4}&-\frac{13}{4}&\frac{1}{2}\\
	5&-\frac{5}{2}&\frac{11}{4}&\frac{11}{4}&-\frac{1}{2}\\
	-7&-\frac{7}{2}&\frac{3}{4}&-\frac{1}{4}&\frac{5}{2}\\
	4&\frac{7}{2}&-\frac{41}{4}&\frac{3}{4}&-\frac{5}{2}
	\end{array}
	\right]
	\left[
	\begin{array}{c}
	y_{1}\\
	y_{2}\\
	y_{3}\\
	y_{4}\\
	y_{5}
	\end{array}
	\right]=
	\left[
	\begin{array}{c}
	-2\\
	3\\
	-3\\
	2
	\end{array}
	\right],
	\]
	where $ \hat{A}_{1}= -2,~ \hat{A}_{2}= \frac{9}{2},~ \hat{A}_{3}=\frac{21}{4},~ \hat{A}_{4}=\frac{1}{4},~ \hat{A}_{5}= -\frac{1}{2},~ \hat{b}= 7 $. The following matrix  $ Q=( \tilde{b}_{i} -\tilde{a}_{ij}) \in M_{4 \times 5}(S) $ is obtained:
	\[
	\left[
	\begin{array}{ccccc}
	0&\boxed{-\frac{9}{2}}&\boxed{-\frac{35}{4}}&\frac{5}{4}&-\frac{5}{2}\\
	\boxed{-2}&\frac{11}{2}&\frac{1}{4}&\frac{1}{4}&\frac{7}{2}\\
	4&\frac{1}{2}&-\frac{15}{4}&\boxed{-\frac{11}{4}}&\boxed{-\frac{11}{2}}\\
	\boxed{-2}&-\frac{3}{2}&\frac{49}{4}&\frac{5}{4}&\frac{9}{2}
	\end{array}
	\right].
	\]
	Since every row of the matrix $ Q $ contains at least one column minimum element, by Theorem~\ref{normthm} the normalized system $\tilde{A}Y=\tilde{b} $ and consequently, the system $ AX=b $ have solutions. Through $ Q $, we can now implement the described method for finding the degrees of freedom of this system:
	\begin{itemize}
		\item \textbf{Step1.} The second and fourth rows of matrix $ Q $ contain exactly one column minimum element, which are both located in the first column. This means $ x_{1} $ is a leading variable of the system $ AX=b $ and therefore $ \mathcal{D}_{f} \leq 5- 1=4 $.
		\item \textbf{Step2.} We must remove every row of $ Q $, which contains the column minimum element in the first column. As a result, the second and fourth rows of $ Q $ are removed. Now, we consider the following submatrix of $ Q $ containing these remaining rows:
		\[
		Q_{r}=\left[
		\begin{array}{ccccc}
		0&\boxed{-\frac{9}{2}}&\boxed{-\frac{35}{4}}&\frac{5}{4}&-\frac{5}{2}\\
		4&\frac{1}{2}&-\frac{15}{4}&\boxed{-\frac{11}{4}}&\boxed{-\frac{11}{2}}
		\end{array}
		\right].
		\]
		\item \textbf{Step3.} Since the column minimum elements in the matrix $ Q_{r} $ have the same frequency, we have four options for the next leading variable. For example, let's consider $ x_{2} $ as a leading variable. Thus, we can remove the first row of $ Q_{r} $. As a result, $ \mathcal{D}_{f} \leq 5- 2 =3$.
		\item\textbf{Step4.} We repeat the process for the  second row of $ Q_{r} $, so the procedure is complete. Consequently, the system under investigation has three leading variables and the number of degrees of freedom is $ \mathcal{D}_{f} = 2 $.
	\end{itemize}
\end{ex}
\begin{ex}
	Consider the linear system $ AX=b $, given in Example~\ref{normex}. In order to find the degrees of freedom of the system $ AX= b $, we must use $ Q $:
	\[
	Q=\left[
	\begin{array}{ccccc}
	\boxed{-117}&21&\boxed{-84}&43&-3\\
	\boxed{-117}&-10&\boxed{-84}&9&-26\\
	-115&\boxed{-49}&\boxed{-84}&10&\boxed{-31}\\
	349&38&252&\boxed{-62}&60
	\end{array}
	\right];
	\]
	The fourth row of $ Q $ contains exactly one column minimum element which is located in the fourth column. $ x_{4} $ is therefore a leading variable of the system $ AX=b $ and the fourth row must be removed from $ Q $. In the remaining rows of $ Q $, the column minimum element in the third column $(-84) $ has the highest frequency, so we choose $ x_{3} $ as the next leading variable of the system $ AX= b $. We now remove every row of $ Q $ containing this column minimum element, so all the rows of $ Q $ are removed. Hence, the system $ AX=b $ has two leading variables and $ \mathcal{D}_{f}=3$. 
\end{ex}
\section{Determining the column rank by normalization method}
We consider the following arbitrary matrix $A$:
\[
A=
\left[
\begin{array}{c|c|c|c}
A_{1}&A_{2}&\cdots&A_{n}
\end{array}
\right],
\]
where $A_{j}$ is the $j$-th column of $A$.\\
We check the existence of solutions of the following system by the normalization method:
\begin{equation}
\label{eq1}
\left[
\begin{array}{c|c|c|c}
A_{1}&A_{2}&\cdots&A_{n-1}
\end{array}
\right]X=A_{n}.
\end{equation}
Here, we have two cases:\\
\begin{enumerate}
	\item If the system $(\ref{eq1})$ has no solution, we conclude that $A_{n}$ is an independent column of $A$. In this case, $A_{n}$ can not be removed from the set of generators of $ Col(A) $. As such, we consider the following system by setting $A_{n}$ as the first column of the coefficient matrix:
	\begin{equation}
	\label{eq2} 
	\left[
	\begin{array}{c|c|c|c|c}
	A_{n}&A_{1}&A_{2}&\cdots&A_{n-2}
	\end{array}
	\right]X=A_{n-1},
	\end{equation}
	\item If the system $(\ref{eq1})$ has solutions, then $A_{n}$ is dependent on the other columns of matrix $A$. Hence, we remove the column $A_{n}$ from the set of generators of $ Col(A) $, and  $colrank(A) \leq n-1$. Now, we can consider the new system as follows.
	\begin{equation}      
	\label{eq3}
	\left[
	\begin{array}{c|c|c|c}
	A_{1}&A_{2}&\cdots&A_{n-2}
	\end{array}
	\right]X=A_{n-1},
	\end{equation}
\end{enumerate}
Next, we check both cases $1$ and $2$ for the systems $(\ref{eq2})$ or $(\ref{eq3})$ depending on which one has happened. We repeat this until we get a linear system whose vector is $A_{1}$ and whose matrix is the independent columns of matrix $A$ which are obtained from the procedure. Finally, we check both cases $1$ and $2$ for this last system. At this point, we can completely determine the independent columns and the column rank of $A$.
\begin{remark}
	Note that we can obtain the row rank of $ A $ by applying the above method to the matrix $ A^{T} $ and finding the column rank of $ A^{T} $, i.e., $ rowrank(A)=colrank(A^{T})$.  
\end{remark}
\begin{ex}
	Consider the following matrix $A \in M_{4 \times 5}(S)$;
	\[
	\left[
	\begin{array}{ccccc}
	4&-4&2&3&3\\
	5&7&7&2&6\\
	10&12&12&8&11\\
	4&-3&2&3&3
	\end{array}
	\right].
	\]
	The normalized matrix $ \tilde{A}=(\tilde{a}_{ij}) $ is
	\[
	\left[
	\begin{array}{ccccc}
	-\frac{7}{4}&-7&-\frac{15}{4}&-1&-\frac{11}{4}\\
	-\frac{3}{4}&4&\frac{5}{4}&-2&\frac{1}{4}\\
	\frac{17}{4}&9&\frac{25}{4}&4&\frac{21}{4}\\
	-\frac{7}{4}&-6&-\frac{15}{4}&-1&-\frac{11}{4}
	\end{array}
	\right];
	\]
	where the $ j$-th column of $ \tilde{A} $ is defined as $ \tilde{A}_{j}= A_{j} - \hat{A}_{j} $, for any $ 1 \leq j\leq 5 $ with $ \hat{A}_{1}= \frac{23}{4},~ \hat{A}_{2}= 3,~ \hat{A}_{3}= \frac{23}{4},~ \hat{A}_{4}= 4,~ \hat{A}_{5}= \frac{23}{4}$.
	In order to determine the column rank of the matrix $ A $, we investigate the linear dependence of each column on other columns of $ A $. First, we solve the following linear system by normalization method:
	\begin{equation}   
	\label{eq4}
	\left[
	\begin{array}{cccc}
	4&-4&2&3\\
	5&7&7&2\\
	10&12&12&8\\
	4&-3&2&3
	\end{array}
	\right]
	\left[
	\begin{array}{c}
	x_{1}\\
	x_{2}\\
	x_{3}\\
	x_{4}
	\end{array}
	\right]=
	\left[
	\begin{array}{c}
	3\\
	6\\
	11\\
	3
	\end{array}
	\right].
	\end{equation}
	By Definition ~\ref{normmethod}, associated with the system $ (\ref{eq4})$, we construct the matrix $ Q_{1}=((Q_{1})_{ij}) \in M_{4}(S) $, where $ (Q_{1})_{ij}=\tilde{a}_{i5} - \tilde{a}_{ij} $, for $ 1 \leq i,j \leq 4 $:
	\[
	Q_{1}=\left[
	\begin{array}{cccc}
	\boxed{-1}&\frac{17}{4}&1&\boxed{-\frac{7}{4}}\\
	1&\boxed{-\frac{15}{4}}&\boxed{-1}&\frac{9}{4}\\
	1&\boxed{-\frac{15}{4}}&\boxed{-1}&\frac{5}{4}\\
	\boxed{-1}&\frac{13}{4}&1&\boxed{-\frac{7}{4}}
	\end{array}
	\right]
	\]
	Due to Theorem~\ref{normthm}, the system $ (\ref{eq4}) $ has solutions, therefore, the fifth column of $ A $ is linearly dependent on the other columns and it must be removed from the set of generators of $ Col(A) $. Hence, $ colrank(A) \leq 4 $.\\
	Next, we repeat this for the following system:
	\begin{equation}
	\label{eq5}     
	\left[
	\begin{array}{cccc}
	4&-4&2\\
	5&7&7\\
	10&12&12\\
	4&-3&2
	\end{array}
	\right]
	\left[
	\begin{array}{c}
	x_{1}\\
	x_{2}\\
	x_{3}
	\end{array}
	\right]=
	\left[
	\begin{array}{c}
	3\\
	2\\
	8\\
	3
	\end{array}
	\right].
	\end{equation}
	The  following matrix $Q_{2}=((Q_{2})_{ij}) \in M_{4 \times 3}(S) $, associated with the system $ (\ref{eq5}) $, can be defined, where $ (Q_{2})_{ij}=\tilde{a}_{i4} - \tilde{a}_{ij} $, for $ 1 \leq i \leq 4 $ and $ 1 \leq j \leq 3$ :
	\[
	Q_{2}=\left[
	\begin{array}{ccc}
	\frac{3}{4}&6&\frac{11}{4}\\
	\boxed{-\frac{5}{4}}&\boxed{-6}&\boxed{-\frac{13}{4}}\\
	-\frac{1}{4}&-5&-\frac{9}{4}\\
	\frac{3}{4}&5&\frac{11}{4}
	\end{array}
	\right]
	\]
	By Theorem~\ref{normthm}, the system $ (\ref{eq5})$ has no solution. Thus, $ A_{4} $ is linearly independent and it cannot be removed from the set of generators of $ Col(A) $ and we consider the following system by setting $ A_{4} $ as the first column of coefficient matrix:
	\begin{equation}
	\label{eq6}
	\left[
	\begin{array}{ccc}
	3&4&-4\\
	2&5&7\\
	8&10&12\\
	3&4&-3
	\end{array}
	\right]
	\left[
	\begin{array}{c}
	x_{1}\\
	x_{2}\\
	x_{3}
	\end{array}
	\right]=
	\left[
	\begin{array}{c}
	2\\
	7\\
	12\\
	2
	\end{array}
	\right].
	\end{equation}
	Here, we can define the matrix $ Q_{3}=(\tilde{a}_{i3} -\tilde{a}_{ij}) \in M_{4 \times 3}(S) $, where $ j=4,1,2$ (respectively) and $ 1 \leq i \leq 4 $ :
	\[
	Q_{3}=\left[
	\begin{array}{ccc}
	\boxed{-\frac{11}{4}}&\boxed{-2}&\frac{13}{4}\\
	\frac{13}{4}&2&\boxed{-\frac{11}{4}}\\
	\frac{9}{4}&2&\boxed{-\frac{11}{4}}\\
	\boxed{-\frac{11}{4}}&\boxed{-2}&\frac{9}{4}
	\end{array}
	\right]
	\]
	Since every row of $ Q_{3} $ contains at least one column minimum element, by Theorem~\ref{normthm}, we conclude that the system $(\ref{eq6})$ has solutions and the column $ A_{3} $ should be removed from the set of generators of $ Col(A) $ and therefore $ colrank(A) \leq 3$. We now consider the following system to investigate linear dependence of the column $ A_{2} $ on the columns $ A_{1}$ and $ A_{4} $:
	\begin{equation}
	\label{eq7}
	\left[
	\begin{array}{cc}
	3&4\\
	2&5\\
	8&10\\
	3&4
	\end{array}
	\right]
	\left[
	\begin{array}{c}
	x_{1}\\
	x_{2}
	\end{array}
	\right]=
	\left[
	\begin{array}{c}
	-4\\
	7\\
	12\\
	-3
	\end{array}
	\right].
	\end{equation}
	Similarly, associated with the system $ (\ref{eq7}) $ we can define the matrix $ Q_{4}=(\tilde{a}_{i2}- \tilde{a}_{ij}) \in M_{4 \times 2}(S) $, where $ j= 4,1 $ (respectively) and $ 1 \leq i \leq 4 $:
	\[
	Q_{4}=\left[
	\begin{array}{cc}
	\boxed{-6}&\boxed{-\frac{21}{4}}\\
	6&\frac{19}{4}\\
	5&\frac{19}{4}\\
	-5&-\frac{17}{4}
	\end{array}
	\right].
	\]
	Clearly, the system $ (\ref{eq7})$ has no solution, so $ A_{2} $ is linearly independent and we can consider the following system by setting $ A_{2} $ as the first column of coefficient matrix:
	\begin{equation}     
	\label{eq8}
	\left[
	\begin{array}{cc}
	-4&3\\
	7&2\\
	12&8\\
	-3&3
	\end{array}
	\right]
	\left[
	\begin{array}{c}
	x_{1}\\
	x_{2}
	\end{array}
	\right]=
	\left[
	\begin{array}{c}
	4\\
	5\\
	10\\
	4
	\end{array}
	\right].
	\end{equation}
	By the normalization method, we can consider the matrix $ Q_{5} =(\tilde{a}_{i1} -\tilde{a}_{ij}) \in M_{4 \times 2}(S) $, where $ 1\leq i \leq 4 $ and $ j= 2, 4$ :
	\[
	Q_{5}=\left[
	\begin{array}{cc}
	\frac{21}{4}&\boxed{-\frac{3}{4}}\\
	\boxed{-\frac{19}{4}}&\frac{5}{4}\\
	\boxed{-\frac{19}{4}}&\frac{1}{4}\\
	\frac{17}{4}&\boxed{-\frac{3}{4}}
	\end{array}
	\right].
	\]
	Obviously, the system $ (\ref{eq8}) $ has solutions. As such, the column $ A_{1} $ is linearly dependent and by removing it from the set of generators of $ Col(A) $, we can conclude that $ colrank(A)= 2$.
\end{ex}
\section{Row-Column Analysis of the System $ AX=b$ }
Consider the system of linear equations $AX=b$, where $A \in M_{m\times n}(S) $, $ b \in S^{m} $ and  $ X $ is an unknown column vector of size $n$. In order to simplify solving the system and determining the degrees of freedom, we reduce the order of the system through a row-column analysis to obtain a new system with fewer equations and variables. We call this new system the reduced system. Note further that the row-column analysis technique is based on the column rank and the row rank of the system matrix, which are determined in the previous section.
\subsection{Column analysis of the system $ AX=b$}
Suppose that $ A_{c_{1}},A_{c_{2}}, \cdots ,A_{c_{n}} $ are the columns of matrix $ A $. Without loss of generality, we can assume that $ A_{c_{1}},A_{c_{2}}, \cdots ,A_{c_{k}}$ are linearly independent and the other columns are linearly dependent on them. The linear system $ AX=b $ can be written as follows.
\begin{center}
	\[
	\left[
	\begin{array}{c|c|c|c|c|c|c}
	A_{c_{1}}&A_{c_{2}}&\cdots&A_{c_{k}}&A_{c_{k+1}}&\cdots&A_{c_{n}}
	\end{array}
	\right]
	\left[
	\begin{array}{c}
	x_{1}\\
	x_{2}\\
	\vdots\\
	x_{k}\\
	x_{k+1}\\
	\vdots\\
	x_{n}
	\end{array}
	\right]
	=
	\left[
	\begin{array}{c}
	b_{1}\\
	b_{2}\\
	\vdots\\
	b_{m}
	\end{array}
	\right].
	\]	
\end{center}
We can rewrite the system as
\begin{equation}
\label{eq9}
\max (A_{c_{1}} + x_{1},A_{c_{2}} + x_{2}, \cdots, A_{c_{k}} + x_{k}, A_{c_{k+1}} + x_{k+1}, \cdots, A_{c_{n}} + x_{n} ) =
\left[
\begin{array}{c}
b_{1}\\
b_{2}\\
\vdots\\
b_{m}
\end{array}
\right].
\end{equation}
There exist scalars $\eta_{ij} \in S $ for every $ 1\leq i\leq k $ and $ k+1\leq j\leq n $ such that 
\begin{equation}      
\label{eq10}
A_{c_{j}} = \max(A_{c_{1}} + \eta_{1j},A_{c_{2}} + \eta_{2j}, \cdots, A_{c_{k}} + \eta_{kj}) .
\end{equation}
By replacing $ (\ref{eq10}) $ in $ (\ref{eq9}) $ we have:
\begin{center}
	$\max (A_{c_{1}} + x_{1},A_{c_{2}} + x_{2}, \cdots, A_{c_{k}} + x_{k}, \max(A_{c_{1}} + \eta_{1(k+1)},A_{c_{2}} + \eta_{2(k+1)}, \cdots, A_{c_{k}} + \eta_{k(k+1)}) + x_{k+1}, \cdots,  \max(A_{c_{1}} + \eta_{1n},A_{c_{2}} + \eta_{2n}, \cdots, A_{c_{k}} + \eta_{kn}) + x_{n} ) =
	\left[
	\begin{array}{c}
	b_{1}\\
	b_{2}\\
	\vdots\\
	b_{m}
	\end{array}
	\right]. $
\end{center}\
Due to the distributivity of $ ``+"$ over $`` \max" $, the following equality is obtained:  
\begin{center}
	$\max [A_{c_{1}} +\max(x_{1}, \eta_{1(k+1)} + x_{k+1},\cdots,\eta_{1n} + x_{n}),A_{c_{2}} +\max(x_{2}, \eta_{2(k+1)} + x_{k+1},\cdots,\eta_{2n} + x_{n}),\cdots,A_{c_{k}} +\max(x_{k}, \eta_{k(k+1)} + x_{k+1},\cdots,\eta_{kn} + x_{n})]=
	\left[
	\begin{array}{c}
	b_{1}\\
	b_{2}\\
	\vdots\\
	b_{m}
	\end{array}
	\right]. $
\end{center}
Now, we can rewrite this system as
\begin{center}
	$\max (A_{c_{1}} + y_{1},A_{c_{2}} + y_{2}, \cdots, A_{c_{k}} + y_{k} ) =
	\left[
	\begin{array}{c}
	b_{1}\\
	b_{2}\\
	\vdots\\
	b_{m}
	\end{array}
	\right],$
\end{center}
where
\begin{equation} 
\label{eq11}
y_{i}= \max(x_{i}, \eta_{i(k+1)} + x_{k+1},\cdots,\eta_{in} + x_{n}),
\end{equation}
for every $ 1\leq i\leq k $. As such, the number of variables decreases from $ n $ to $ k $.
Next, we show that the existence of solutions of the system $ AX=b $ depends on the row rank of $ A $. Assume that $ Y^{*} =(y^{*}_{i})_{i=1}^{k} $ is the maximal solution of the system: 
\begin{center}
	\[
	\left[
	\begin{array}{c|c|c|c}
	A_{c_{1}}&A_{c_{2}}&\cdots&A_{c_{k}}
	\end{array}
	\right]
	\left[
	\begin{array}{c}
	y^{*}_{1}\\
	y^{*}_{2}\\
	\vdots\\
	y^{*}_{k}
	\end{array}
	\right]
	=
	\left[
	\begin{array}{c}
	b_{1}\\
	b_{2}\\
	\vdots\\
	b_{m}
	\end{array}
	\right].
	\]
\end{center}
Hence, the equalities $ (\ref{eq11})$ imply the system $ AX=b $ should have solutions $ x_{j}\leq \min (y^{*}_{1} -\eta_{1j},y^{*}_{2} -\eta_{2j}, \cdots,y^{*}_{k} -\eta_{kj}) $ for every $ k+1 \leq j \leq n $ and $ x_{i} = y^{*}_{i}  $ for every $ 1 \leq i \leq k $.
\subsection{Row analysis of the system $ AX =b $}
Consider the system $ AX =b $ in the form of
\begin{equation}
\label{eq12}
\left[
\begin{array}{c}
A_{r_{1}}\\ \hline
A_{r_{2}}\\ \hline
\vdots\\ \hline
A_{r_{h}}\\ \hline
A_{r_{h+1}}\\ \hline
\vdots \\\hline
A_{r_{m}}\\
\end{array}
\right]
\left[
\begin{array}{c}
x_{1}\\
x_{2}\\
\vdots\\
x_{n}
\end{array}
\right]
=
\left[
\begin{array}{c}
b_{1}\\
b_{2}\\
\vdots\\
b_{h}\\
b_{h+1}\\
\vdots \\
b_{m}
\end{array}
\right]
\end{equation}
where $ A_{r_{i}} $ is the $ i $-th row of the matrix $ A $, for every $ 1 \leq i \leq m $. Without loss of generality, the rows $ A_{r_{1}}, A_{r_{2}}, \cdots, A_{r_{h}} $ can be considered linearly independent rows of $A$ and the other rows $ A_{r_{i}} $, $ h+1 \leq i \leq m $ are linear combinations of them. Consequently, there exist scalars $ \xi _{ij} \in S $ for every $ 1 \leq j \leq h $ and $ h+1 \leq i \leq m $ such that:
\begin{equation}
\label{eq13}
A_{r_{i}}= \max (A_{r_{1}} + \xi_{i1}, A_{r_{2}} + \xi_{i2}, \cdots, A_{r_{h}} + \xi_{ih}) 
\end{equation}
, for every $ h+1 \leq i \leq m $. We can now rewrite the system of equations $(\ref{eq12}) $ as 
\begin{center}
	$A_{r_{i}}
	\left[
	\begin{array}{c}
	x_{1}\\
	x_{2}\\
	\vdots\\
	x_{n}
	\end{array}
	\right]= b_{i} $, for any $ 1 \leq i \leq m$,
\end{center} 
which can become the $ h $-equation system:
\begin{center}
	$A_{r_{j}}
	\left[
	\begin{array}{c}
	x_{1}\\
	x_{2}\\
	\vdots\\
	x_{n}
	\end{array}
	\right]= b_{j} $, for any $ 1 \leq j \leq h$.
\end{center}
We now obtain the row-reduced system with $h$ equations.\\
Note that in the process of reducing the system $AX=b$, it does not matter which of the row or column analysis is first applied to the system.\\
This argument leads us to investigate the existence of solutions of the linear system $ AX=b $.
\begin{theorem}\label{SEOVER}
	Let $A \in M_{m \times n}(S)$. The system $AX=b$ has solutions if and only if its reduced system, $\overline{A}Y=\overline{b}$, has solutions. 
\end{theorem}
\begin{proof}
	Let $ colrank(A)=k$ and $rowrank(A)=h $. By applying row-column analysis on the system $AX=b$ and replacing $(\ref{eq13})$ in the $ m $-equation system $(\ref{eq12})$, we conclude that
	\begin{equation}   
	\label{eq14}
	b_{i}=
	A_{r_{i}}
	\left[
	\begin{array}{c}
	x_{1}\\
	x_{2}\\
	\vdots\\
	x_{n}
	\end{array}
	\right]= \max(b_{1} + \xi_{i1}, b_{2} + \xi_{i2}, \cdots,b_{h} + \xi_{ih}),
	\end{equation}
	for every $ h+1 \leq i \leq m $. If the equalities $(\ref{eq14})$ hold for every $ h+1 \leq i \leq m $, then we can reduce the system $ AX=b $ to the system $\overline{A}Y=\overline{b}$, where $ \overline{A}$ is the reduced $ h\times k $ matrix obtained from $ A $, $Y $ is an unknown vector of size $ k $, and $ \overline{b}$ is the reduced vector obtained from $ b $. Thus, the existence of solution $AX=b$ and $\overline{A}Y=\overline{b}$ depends on each other.
\end{proof}
The next theorem shows that this technique for row-column analysis simplifies the computation of degrees of freedom. 
\begin{theorem}
	Let $A \in M_{m \times n}(S)$. If the system $AX=b$ has solutions, then the number of degrees of freedom is $k-p$, where $k$ is the column rank of $A$ and $p$ is the number of columns of matrix $\overline{A}$ such that $\overline{b}$ is the linear combination of these columns.
\end{theorem}
\begin{proof}
	By theorem~\ref{SEOVER}, the existence of solutions of the system $AX=b $ and $\overline{A}Y=\overline{b}$ depend on each other. As such, their degrees of freedom are equal, and the proof is complete.
\end{proof}
\begin{remark}
	Note that if $\overline{b}$ is not a linear combination of any column of $\overline{A}$, then the systems $AX=b$ and $\overline{A}Y=\overline{b}$ have no solutions.
\end{remark}
\section{Concluding Remarks}\label{remarks}
In this paper, applying the normalization method to a linear system, we presented a necessary and sufficient condition for the system to have a solution. In order to determine the degrees of freedom of a solvable system, the column rank and the row rank of an arbitrary matrix, some procedures were proposed as well. We also used a row-column analysis technique to reduce the order of the linear systems over tropical semirings and simplify their solution process.
	
	\newpage
\appendix
\begin{center}
	\includegraphics[scale=.7]{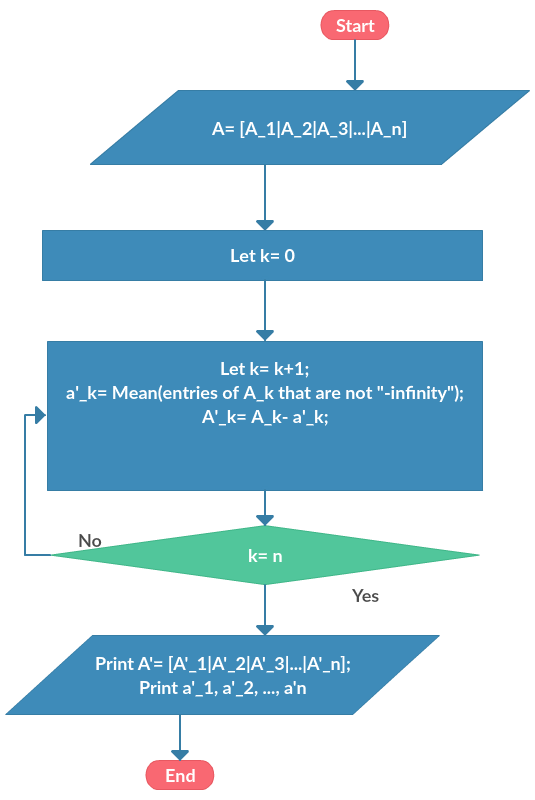}\\
	\title{\textbf{Figure 1}: Matrix Normalization}
\end{center}
\begin{center}
	\includegraphics[scale=.5]{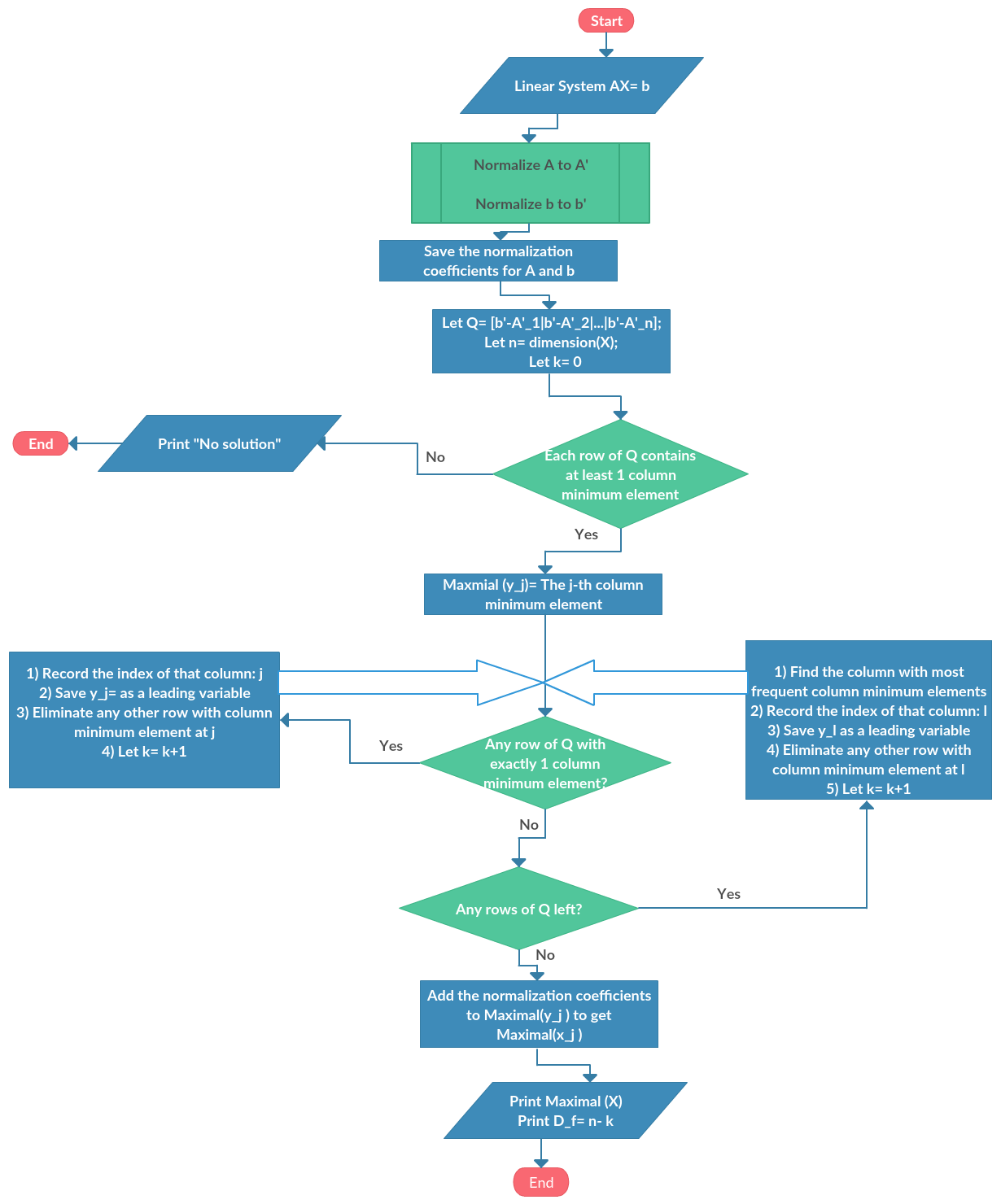}\\
	\title{\textbf{Figure 2}: Linear System Solution and Degrees of Freedom through Normalization}
\end{center}
\begin{center}
	\includegraphics[scale=.5]{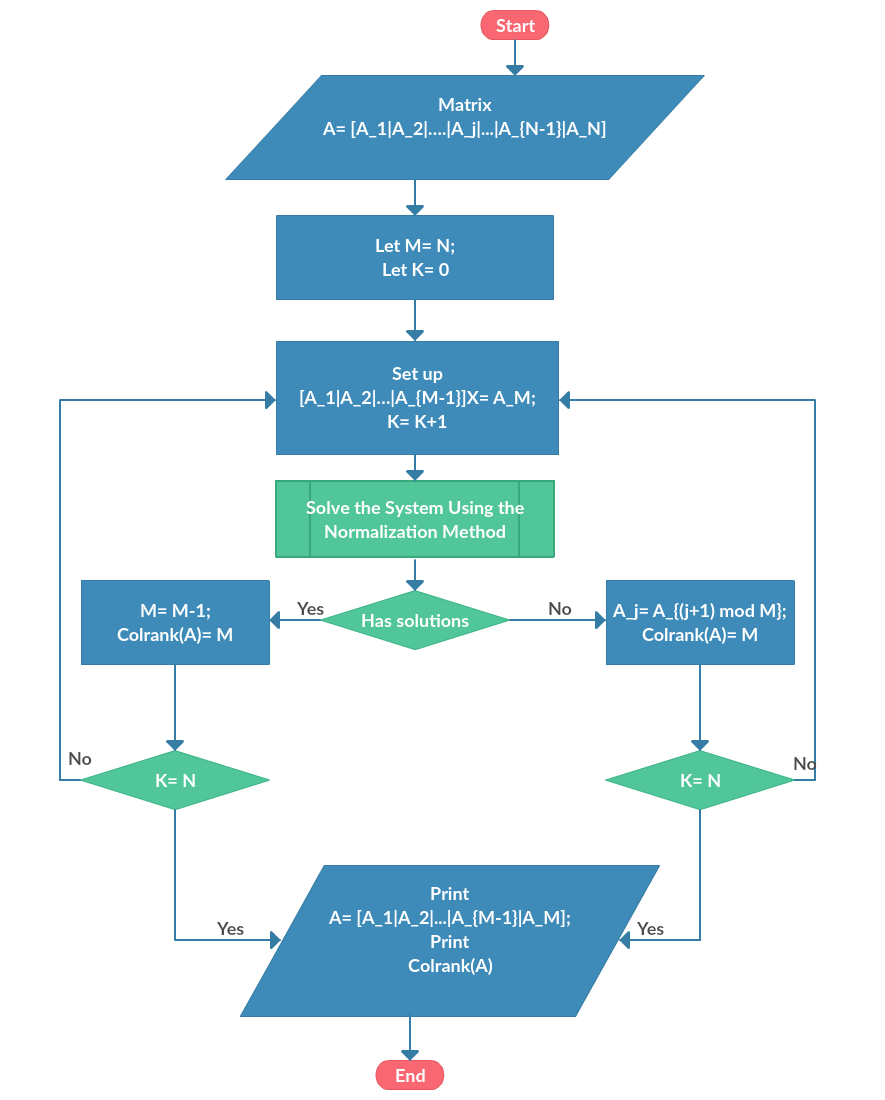}\\
	\title{\textbf{Figure 3}: Column Rank of A Matrix}
\end{center}
\end{document}